\documentclass{amsart}
\usepackage{todonotes}
\usepackage{stackrel}

\usepackage{xypic}
\usepackage{mathtools}

\DeclarePairedDelimiter\floor{\lfloor}{\rfloor}

\usepackage{graphicx, amsmath, amssymb, amsthm}

\usepackage{tikz}
\usepackage{hyperref} 

\usetikzlibrary{arrows,shapes}
\usetikzlibrary{calc,fadings,decorations.pathreplacing}

\newcommand\pgfmathsinandcos[3]{%
  \pgfmathsetmacro#1{sin(#3)}%
  \pgfmathsetmacro#2{cos(#3)}%
}
\newcommand\LongitudePlane[3][current plane]{%
  \pgfmathsinandcos\sinEl\cosEl{#2} 
  \pgfmathsinandcos\sint\cost{#3} 
  \tikzset{#1/.estyle={cm={\cost,\sint*\sinEl,0,\cosEl,(0,0)}}}
}

\newcommand\DrawLongitudeCircle[2][1]{
  \LongitudePlane{\angEl}{#2}
  \tikzset{current plane/.prefix style={scale=#1}}
  \pgfmathsetmacro\angVis{atan(sin(#2)*cos(\angEl)/sin(\angEl))} %
  \draw[current plane] (\angVis:1) arc (\angVis:\angVis+180:1);
  \draw[current plane,dashed] (\angVis-180:1) arc (\angVis-180:\angVis:1);
}

\newcommand{\mycases}[1]{\left\{\begin{array}{ll}#1\end{array}\right.}
\newcommand{\Z}{{\bf  Z}}
\newcommand{\MU}{MU}
\newcommand{\smashove}[1]{\underset{#1}{\wedge}}

\newcommand{\Ind}{\big\uparrow} 
\newcommand{\Res}{\big\downarrow} 
\DeclareMathOperator{\Hom}{Hom}

\newcommand{\res}{res}
\newcommand{\tr}{tr}

\newcommand{\Cp}[1]{C_{p^{#1}}}
\newcommand{\Cpn}{\Cp{n}}
\newcommand{\cp}[1]{p^{#1}}

\newcommand{\m}[1]{{\protect\underline{#1}}}
\newcommand{\mZ}{\m{\Z}}
\newcommand{\mM}{\m{M}}

\newcommand{\mB}{\m{B}}

\mathchardef\mhyphen="2D
 
\newcommand{\EM}{Eilenberg-Mac~Lane}

\newtheorem{theorem}{Theorem}[section]
\newtheorem{thm}[theorem]{Theorem}
\newtheorem{lemma}[theorem]{Lemma}
\newtheorem{corollary}[theorem]{Corollary}
\newtheorem{cor}[theorem]{Corollary}
\newtheorem{definition}[theorem]{Definition}
\newtheorem{defin}[theorem]{Definition}
\newtheorem{defn}[theorem]{Definition}
\newtheorem{proposition}[theorem]{Proposition}
\newtheorem{prop}[theorem]{Proposition}

\newtheorem{remark}[theorem]{Remark}

\newtheorem{notation}[theorem]{Notation}

\newcommand{\rep}{representation}

\begin{document}
 
\title[Slices and $RO(G)$-graded Suspensions I]{The Slice Spectral
Sequence for certain $RO({\Cpn} )$-graded Suspensions of $H\mZ$}

\author{M.~A.~Hill}
\address{Department of Mathematics \\ University of California, Los Angeles
\\Los Angeles, CA 90095}
\email{mikehill@math.ucla.edu}

\author{M.~J.~Hopkins}
\address{Department of Mathematics \\ Harvard University
\\Cambridge, MA 02138}
\email{mjh@math.harvard.edu}

\author{D.~C.~Ravenel}
\address{Department of Mathematics \\ University of Rochester 
\\Rochester, NY}
\email{doug@math.rochester.edu}

\thanks{M.~A.~Hill was partially supported by NSF grants DMS-0905160 , DMS-1307896 and
the Sloan foundation}
\thanks{M.~J.~Hopkins was partially supported the  NSF grant
DMS-0906194}
\thanks{D.~C.~Ravenel was partially supported by the NSF grants DMS-1307896 and  DMS-0901560}
\thanks{All three authors received support from the DARPA grants
HR0011-10-1-0054-DOD35CAP and FA9550-07-1-0555}

\begin{abstract}
We study the slice filtration and associated spectral sequence for a family of $RO(C_{p^{n}})$-graded suspensions of the Eilenberg-MacLane spectrum for the constant Mackey functor $\underline{\mathbb Z}$. Since $H\underline{\mathbb Z}$ is the zero slice of the sphere spectrum, this begins an analysis of how one can describe the slices of a suspension in terms of the original slices. 
\end{abstract}
 
\maketitle

\section{Introduction}
A key ingredient to our solution to the Kervaire invariant one problem \cite{HHR} was a new equivariant tool: the slice filtration. Generalizing the $C_{2}$-equivariant work of Dugger
\cite{Dugger} and modeled on Voevodsky's motivic slice filtration
\cite{Voe:Open}, this is an exhaustive, strongly convergent
filtration on genuine $G$-spectra for a finite group $G$. While quite
powerful, the filtration quotients are quite difficult to determine,
and much of our solution revolved around determining them for various
spectra built out of $\MU$.

Classically one has a map from a spectrum $X$ to its
$n$\textsuperscript{th} Postinikov section $P^{n}X$.  The latter is
obtained from $X$ by attaching cells to kill off all homotopy groups
above dimension $n$, and the tower can be thought of as assembling $X$
by putting in the homotopy groups one at a time. This process can also
be described as a localization or nullification functor, where we
localize by killing the subcategory of $n$-connected spectra. Though
this is a very big subcategory, it is generated by a much smaller set:
the spheres of dimension at least $(n+1)$.

In the equivariant context, there are two axes along which we can vary
this construction:
\begin{enumerate}
\item there are more spheres, namely representation spheres for real
representations of $G$ and
\item there are subgroups of $G$, the behavior for each of which can
be controlled.
\end{enumerate}
The equivariant Postnikov filtration largely ignores the first option,
choosing instead to build a localization tower killing the
subcategories generated by all spheres (with a trivial action!) of
dimension at least $(n+1)$ for all subgroups of $G$. Equivalently, we
nullify the subcategory generated by all spectra of the form
\[
G_{+}\wedge_{H}S^{m}
\]
for $m>n$.

The slice filtration takes more seriously the option of blending the
representation theory and the homotopy theory.
\begin{defn}\label{def-filt}
For each integer $n$, let $\tau_{\geq n}$ denote the localizing
subcategory of $G$-spectra generated by
\[
G_{+}\smashove{H}S^{k\rho_{H}-\epsilon },
\] 
where $H$ ranges over all subgroups of $G$, $\rho_{H}$ is the regular
representation of $H$, $k|H|-\epsilon \geq n$ and $\epsilon =0,1$.

If $X$ is an object of $\tau_{\geq n}$, we say that $X$ is slice
greater than or equal to $n$ and that $X$ is slice $(n-1)$-connected.

The associated localization tower:
\[
X\to P^{\ast}X
\]
is the slice tower of $X$.
\end{defn}

The slice filtration and the slice tower refine the non-equivariant
Postnikov tower, in the sense that forgetting the $G$-action takes the
slice tower to the Postnikov tower, but the equivariant layers of the
slice tower are in general much more complicated. In particular, they
need not be {\EM} spectra.

This paper is the start of a short series analyzing several curious
points of the slice filtration for cyclic $p$-groups applied to a very
simple yet interesting family: the suspensions of the {\EM} spectrum
$H\mZ$ by virtual representation spheres. This paper will discuss the
suspensions by multiples of the irreducible faithful representation
$\lambda$. This compliments the work of Yarnall on the ordinary suspensions of $H\m{\Z}$ \cite{Yarnall}.

We make a huge, blanket assumption and a slight abuse of
notation. {\em{Everything}} we consider is localized at $p$, where $p$
is the prime dividing the order of the group. Thus when we write
things like $\Z$, we actually mean $\Z_{(p)}$. This being said, we
have tried as much as possible to make integral statements which are
prime independent (specifying actual representations, rather than
$J$-equivalence classes whenever possible). Additionally, essentially
everything we say holds for odd primes. The case of $p=2$ behaves
quite differently, given the existence of a non-trivial
$1$-dimensional real representation, the sign representation.

\subsection{Names of representations} Fix an identification of ${\Cpn}
$ with the $p^{n}$th roots of unity $\mu_{p^{n}}$. For all $k\in \Z$,
let $\lambda(k)$ denote the composite of the inclusion of these roots
of unity with the degree $k$-map on $S^{1}$. This is a representation
of complex dimension $1$.

The real regular representation will be denoted $\rho$. We have a splitting
\[
\rho=1+\bigoplus_{j=1}^{\tfrac{p^{n}-1}{2}}\lambda(j)
\]
for $p>2$. We will also let $\bar{\rho}$ denote the reduced regular
representation, the quotient of the regular representation by the
trivial summand.

\subsection{Representation Spheres} Representation spheres for
$C_{p^{n}}$ have an exceptionally simple cell structure. This renders
computations much more tractable than one might expect. The key
feature is that the subgroups of $C_{p^{n}}$ are linearly ordered, and
we can use this to build significantly smaller cell structures that
might be expected. Our analysis follows \cite{HHR} and
\cite{CDMProof}.

We first consider the sphere $S^{\lambda(p^{k})}$. A cell structure is
given by rays from the origin through the roots of unity, together
with the sectors between these one cells. This gives a cell structure
\[
S^{0}\cup C_{p^{n}}/C_{p^{k}+}\wedge e^{1}\cup
C_{p^{n}}/C_{p^{k}+}\wedge e^{2}.
\]
A picture is given in Figure~\ref{fig:Cellmufour} for $\lambda(1)$ and
$C_{8}$.

\begin{figure}[h]
\begin{tikzpicture}[scale=1]
\def\R{1.5} 
\def\angEl{35} 
\filldraw[ball color=white] (0,0) circle (\R);
\foreach \t in {0,-45,...,-135} { \DrawLongitudeCircle[\R]{\t} }
\end{tikzpicture}
\caption{}
\label{fig:Cellmufour}
\end{figure}
 
We can smash these together to get a cell structure on $S^{V}$ for any
[virtual] representation $V$. The simplification arises from the
different stabilizers which arise. For all $m\leq k$, the restriction
$i_{C_{p^{m}}}^{\ast}\lambda(p^{k})$ is trivial. Thus
\begin{align*}
S^{\lambda(p^{m})+\lambda(p^{k})}
 &\cong S^{\lambda(p^{m})}\wedge S^{\lambda(p^{k)}}\\
 &\cong S^{\lambda(p^{m})}\wedge S^{\lambda(p^{k)}}\\
 &\cong S^{\lambda(p^{k})} \cup C_{p^{k}}/C_{p^{m}+}\wedge 
              e^{3}\cup C_{p^{k}}/C_{p^{m}+}\wedge e^{4}.
\end{align*}

By induction, this shows how to build any representation sphere:
decompose $V$ into irreducibles, sort these in order of decreasing
stabilizer subgroup, and repeat the above trick. As an aside, this
also allows a determination of a cell structure for virtual
representation spheres, where we approach them the same way.

Our discussion of cell structure focused primarily on that of
$\lambda(p^{k})$. This is because equivariant homotopy groups are
actually more honestly called $JO(G)$-graded: they depend not on the
representation but rather on the homotopy type of the associated
sphere. In the $p$-local context, most of the irreducible
representations for ${\Cpn} $ have equivalent one point
compactifications: if $r$ is prime to $p$, then
\[
S^{\lambda(rp^{k})}\simeq S^{\lambda(p^{k})}.
\]

For this reason, we can single out a single $JO$-equivalence class:
let $\lambda_{k}=\lambda(p^{k})$. {\em We will denote $\lambda_{0}$ 
simply by $\lambda $.} As $k$ varies, the associated representation
spheres hit every $JO$-equivalence class.

We summarize the above discussion in a simple proposition.

\begin{prop}\label{prop:CellStructure}
If
\[
V=k_{n}+k_{n-1}\lambda_{n-1}+\dots+k_{0}\lambda_{0},
\]
then
\begin{multline*}
S^{V}=S^{k_{n}}\cup {\Cp{n}/\Cp{n-1}}_{+}\wedge e^{k_{n}+1}\underset{1-\gamma}{\cup} {\Cp{n}/\Cp{n-1}}_{+}\wedge e^{k_{n}+2}\cup \dots \cup \\ {\Cp{n}/\Cp{n-1}}_{+}\wedge e^{k_{n}+2k_{n-1}-1}
\underset{1-\gamma}{\cup} {\Cp{n}/\Cp{n-1}}_{+}\wedge e^{k_{n}+2k_{n-1}} \cup \\ {\Cp{n}/\Cp{n-2}}_{+}\wedge e^{k_{n}+2k_{n-1}+1}\cup\dots \underset{1-\gamma}{\cup} {\Cp{n}}_{+}\wedge e^{\dim V}, 
\end{multline*}
where $\gamma$ is a generator of $\Cp{n}$.
\end{prop}

\subsection{The slice filtration} 

We quickly recall several important
facts about the slice filtration. This section contains no new
results, and all proofs can be found in \cite{Hill:Primer} and in
\cite[\S3-4]{HHR}.

We first connect the slice filtration and the Postnikov filtration.

\begin{proposition}
If $X$ is $(n-1)$-connected for $n\geq 0$, then $X$ is in $\tau_{\geq n}$.
\end{proposition}

The converse of this is visibly not true, as generators of the form
$S^{k\rho_{G}}$ are not $(k|G|-1)$-connected for $G\neq\{e\}$.

A slight elaboration on this lets us generalize this for smash products.
\begin{proposition}
If $X$ is $(-1)$-connected and $Y$ is in $\tau_{\geq n}$, then
$X\wedge Y$ is in $\tau_{\geq n}$.
\end{proposition}
Unfortunately, smashing reduced regular representation spheres for $G$
shows that results of this form are the best possible.

The form of the generating spectra for $\tau_{\geq n}$ shows that the
slice tower commutes with suspensions by regular representation
spheres.

\begin{proposition}
For any spectrum $X$,
\[
P^{k}\Sigma^{\rho}X\cong\Sigma^{\rho}P^{k-|G|}X,
\]
and identically for slices.
\end{proposition}

The final feature we will need is a recipe for determining for which
$n$ a particular representation sphere $S^{V}$ is in $\tau_{\geq
n}$. This is restatement of \cite[Corollary~3.9]{Hill:Primer},
specializing the Corollary there to the case $Y=S^{0}$.

\begin{proposition}\label{prop:SliceConnectivity}
If $W$ is a representation of $G$ such that $S^{W}$ is in $\tau_{\geq
\dim W}$ and $V$ is a sub representation such that
\begin{enumerate}
\item the inclusion induces an equality $V^{G}=W^{G}$ and
\item for all proper subgroups $H$, the restriction
$i_{H}^{\ast}S^{V}$ is in $\tau_{\geq \dim V}$,
\end{enumerate}
then $S^{V}$ is in $\tau_{\geq\dim V}$. 
\end{proposition}

The conditions are very easy to check, by induction on the order of
the group. In general, we will alway choose $W$ to be a regular
representation or a reduced regular representation, as these are
guaranteed to be in the correct localizing categories.

\section[Dramatis personae: some special Z-modules]{Dramatis personae:
some special \texorpdfstring{${\mZ}$}{Z}-modules}\label{sec-drama}
  
A large number of Mackey functors will show up in our analysis. They
are all variants of the constant Mackey functor $\mZ$ and the Mackey
functor $\mB$ that is the ``Bredon homology Mackey functor''. Many of
our discussions are facilitated by pictures of Mackey functors, and
following Lewis, we draw them vertically. A generic Mackey functor $\mM$ for  
$C_{p}$ will be drawn

\[
\xymatrix{{\mM(G/G)}\ar@(l,l)[d]_{\res} \\
{\mM(G/e)} \ar@(r,r)[u]_{\tr} \ar@(dl,dr)[]_{\gamma} }
\]
where $\res$ is the restriction map, $\tr$ is the transfer, and $\gamma$
generates the Weyl group. For larger cyclic groups, we will use the
obvious extension of this notation. Moreover, if the Weyl action is
trivial or obvious, then we will suppress the map $\gamma$.

\subsection{Forms of \texorpdfstring{$\mZ$}{Z}} 
The constant Mackey functor $\mZ$ (in
which all restriction maps are the identity are the transfers are
multiplication by the index) and its dual (in which the roles of
restriction and transfer are reversed) are just two members of a
family of distinct Mackey functors which take the value
$\mZ$ on each orbit $G/H$. For ${\Cpn} $, there are $2^{n}$ distinct
Mackey functors, corresponding to the $n$ choices ``is the restriction
$1$ or $p$'' for adjacent subgroups.

A large subfamily of these occur in our discussion, and we give some
notation here.
\begin{definition}
Let $0\leq j < k\leq n$ be integers. For each pair, let $\mZ(k,j)$
denote the Mackey functor with constant value $\Z$ for which the
restriction maps are
\[
res_{p^{s}}^{p^{s+1}}=
res_{C_{p^{s}}}^{C_{p^{s+1}}}=
\begin{cases}
1 & s < j, \\
p & j \leq s < k, \\
1 & k\leq s.
\end{cases}
\]
\end{definition}

{\em From now on, when a cyclic group appears as an index, we will
abbreviate it by the order of the group as we did above.}

Thus the restriction of $\mZ(k,j)$ to ${\Cp{j}} $ is just the constant
Mackey functor $\mZ$. For subgroups of ${\Cp{k}} $ which properly
contain ${\Cp{j}} $, it looks like the dual to the constant Mackey
functor $\mZ$, and for those which properly contain ${\Cp{k}} $, it
looks again like $\mZ$.

Equivalently, $\mZ(k,j)$ is the unique $\Z$-valued Mackey functor
$C_{p^{n}}$ in which $\res_{1}^{p^{j}}$, $\tr_{p^{j}}^{p^{k}}$ and
$\res_{p^{k}}^{p^{n}}$ are isomorphisms.

While $\mZ^{\ast}=\mZ(n,0)$, $\mZ$ does not occur in our list. If we
allow $j=k$, then $\mZ=\mZ(j,j)$ for any choice of $j$. For the
reader's convenience, we draw out the four variants of $\mZ$ that
occur for $C_{p^{2}}$.

\[
\xymatrix@!C=.85in{
{\mZ} 
    & {\mZ(1,0)} 
        & {\mZ(2,1)} 
            & {\mZ(2,0)} \\ 
{\Z}\ar@(l,l)[d]_{1} 
    & {\Z}\ar@(l,l)[d]_{1} 
        & {\Z}\ar@(l,l)[d]_{p} 
            & {\Z}\ar@(l,l)[d]_{p} \\
{\Z} \ar@(l,l)[d]_{1} \ar@(r,r)[u]_{p} 
    & {\Z}\ar@(l,l)[d]_{p}\ar@(r,r)[u]_{p} 
         & {\Z}\ar@(l,l)[d]_{1}\ar@(r,r)[u]_{1} 
            & {\Z} \ar@(r,r)[u]_{1} \ar@(l,l)[d]_{p}\\
{\Z}\ar@(r,r)[u]_{p} 
    & {\Z}\ar@(r,r)[u]_{1} 
         & {\Z}\ar@(r,r)[u]_{p} 
            & {\Z}\ar@(r,r)[u]_{1}
}
\]

In our analysis of the slices, we will need some elementary
computations with $\mZ(k,j)$. We will eventually produce a projective
resolution of all of these, but for now, we shall content ourselves to
$\mZ(n,k)$. 

First a brief digression on induced and restricted Mackey functors.   

\begin{defin}\label{def-IndRes}
Let $\mM$ and $\m{N}$ be Mackey functors for abelian groups $G$ and
$H\subset G$ respectively, and let $i_{H}^{*}$ denote the forgetful functor
from $G$-sets to $H$-sets.  Then the induced Mackey functor
$\Ind_{H}^{G}(\m{N})$ on $G$ and the restricted Mackey functor
$\Res_{H}^{G}(\mM)$ on $H$ are given by
\begin{align*}
\Ind_{H}^{G}\m{N} (G/K)
 & = \m{N} (i_{H}^{*}G/K)  \\
 & = \mycases{
\Z[G]\otimes_{\Z[H]}  \m{N} (H/K)
       &\mbox{for }K\subseteq H\\
(\Z[G]\otimes_{\Z[H]}  \m{N} (H/H))^{K}
       &\mbox{for }H\subseteq K
}  \\ 
\Res_{H}^{G}\mM (H/K)
 & = \mM (G/K)\qquad \mbox{for }K\subseteq H.
\end{align*}

\noindent The Weyl action of $G$ in $\Ind_{H}^{G}\m{N}$ is induced by
the $H$-action on $\m{N}$ and the $G$-action on $G/K$. The Weyl action
of $H$ in $\Res_{H}^{G}\mM$ is the restriction of the the $G$-action
on $\mM$.
\end{defin}

Note that the second description of $\Ind_{H}^{G}\m{N}$ is
not complete for general $G$ since there are subgroups $K$ that
neither contain nor are contained in $H$.  However it is complete when
$G$ is a cyclic $p$-group, the case of interest here.  

In particular if $\m{N}$ is the fixed point Mackey functor for a
$\Z[H]$-module $N$ defined by $\m{N} (H/K) = N^{K}$, then 
\[
\Ind_{H}^{G}\m{N} = \underline{\Z[G]\otimes_{\Z[H]}N},
\]
the fixed point Mackey functor for the
$\Z[G]$-module $\Z[G]\otimes_{\Z[H]}N$.

We now return to cyclic $p$-groups.  For the constant Mackey
functor $\mZ$ on $G=C_{p^{n}}$, the composite of induction and the
restriction to ${\Cp{k}} $ is the fixed point Mackey functor for the
${\Cpn} $-module $\Z[{\Cpn} /{\Cp{k}} ]$: \[ \Ind_{p^{k} }^{p^{n}
}\Res_{p^{k} }^{p^{n} }\big(\mZ\big) \cong \m{\Z[{\Cpn} /{\Cp{k}} ]}.
\]
As such, it is very easy to describe maps out of this Mackey functor
in the category of $\mZ$-modules:
\[
\Hom_{\mZ}(\Ind_{p^{k} }^{p^{n} }\Res_{p^{k} }^{p^{n} }(\mZ),\mM)
  \cong\Hom_{\downarrow_{{p^{k}}}^{{p^{n}}}\mZ}(\Res_{{p^{k}}}^{{p^{n}}}\mZ,
                   \Res_{p^{k} }^{p^{n} }\mM)
       \cong \mM({\Cpn} /{\Cp{k}} )
\]
for a Mackey functor $\mM$ on $G$.  In particular, these
are all projective objects in the category of $\mZ$-modules.

These observations determine the maps in the following proposition.

\begin{prop}\label{prop-2.3}
The sequence
\[
0\to\mZ\to\m{\Z[{\Cpn} /{\Cp{k}} ]}\xrightarrow{1-\gamma}
    \m{\Z[{\Cpn} /{\Cp{k}} ]}\to
        \mZ(n,k)\to 0
\]
is exact, where $\gamma$ denotes a generator of ${\Cpn} $. The map
\[
\m{\Z[{\Cpn} /{\Cp{k}} ]}\to\mZ(n,k)
\]
is left-adjoint to the identity map from $\mZ$ to the restriction to
${\Cp{k}} $ of $\mZ(n,k)$.

The first three terms are a projective resolution of $\mZ(n,k)$ in the
category of $\mZ$-modules.
\end{prop}

\begin{proof}
It is obvious that $\mZ$ is the kernel of the map $1-\gamma$. For
subgroups of ${\Cp{k}} $, the fixed point Mackey functors are a direct
sum of copies of $\mZ$, and $\gamma$ acts by permuting the
summands. Thus the quotient Mackey functor is also the constant Mackey
functor $\mZ$ for subgroups of ${\Cp{k}} $.

When we look at subgroups of ${\Cpn} $ which contain ${\Cp{k}} $ then
we are actually looking at the Mackey functor for the integral regular
representation of the group ${\Cpn} /{\Cp{k}} $. The fixed points are
given by various transfers, and the fixed points are the obvious
inclusions. Passing to the quotient by $(1-\gamma)$ then sets all
these inclusion maps to multiplication by the index. The transfer maps
are the identity.
\end{proof}

\noindent A useful way to understand these is that we have for
${\Cp{k}} $ the value $\Z^{p^{n-k}}$. The restriction maps are
diagonals, and the transfers are fold maps. Then the statement about
passing to a quotient is obvious.

\subsection{Forms of \texorpdfstring{$\mB$}{B}}

The Hom groups between various $\mZ(k,j)$ are all easy to work out
(the associated $Hom$ groups are all $\Z$). We shall encounter maps
from $\mZ(k,j)$ to $\mZ$, and it is not difficult to see that the maps
are parameterized by where the element $1$ in $\mZ(k,j)(G/e)$ goes.

\begin{definition}
Let $\mB_{k,j}$ denote the quotient Mackey functor associated to the
inclusion $\mZ(k+j,j)\to \mZ$ which is an isomorphism when evaluated on
$G/e$.  Equivalently, $\mB_{k,j}$ is the quotient of the unique map
\[
\mZ(k+j,0)\to \mZ(j,0)
\]
which is the identity when restricted to $\Cp{j}$.

Let $\mB_{k,j}^{*}$ be the quotient of the unique map
\[
\mZ(n,j)\to \mZ(n,k+j)
\]
which is the identity when restricted to $\Cp{j}$.
\end{definition}

It will be helpful to allow $k=0$ in the above definition, in which
case $\mB_{0,j}$ is the zero Mackey functor.

The Mackey functor $\mB_{k,j}$ is simple to describe: 
\[
\mB_{k,j}({\Cpn} /{\Cp{m}} )=\begin{cases}
\Z/p^{k} & m \geq k+j, \\
\Z/p^{m-j} & j < m < k+j, \\
0          & m \leq j.
\end{cases}
\]
All restriction maps are the canonical quotients, and the transfers
are multiplication by $p$.  Thus $j$ refers to the largest subgroup
for which $\mB_{k,j}$ restricts to $0$, while $k$ indicates the
order of the maximal $p$-torsion present.

Here is a partial Lewis diagram illustrating the definitions of
$\mB_{k,j}$ and $\mB_{k,j}^{*}$ for $G=C_{p^{n}}$, where $m
=n-k-j$, $0\leq s\leq k$ and $s'=k-s$.  The horizontal sequences are
short exact.

\begin{displaymath}
\xymatrix
@R=3mm
@!C=10mm
{
{\mM }
    &{\mZ (k+j,0)}\ar[r]
        &{\mZ (j,0)}\ar[r]
            &{\mB_{k,j}}
                &{\mB_{k,j}^{*}}
                    &{\mZ (n,k+j)}\ar[l]
                        &{\mZ (n,j)}\ar[l]\\
{\mM (G/G)}
    &\Z \ar[r]^(.5){p^{k}} \ar@/_.75pc/[d]_(.5){1}
        &\Z \ar[r]\ar@/_.75pc/[d]_(.5){1}
            &\Z/p^{k} \ar@/_.75pc/[d]_(.5){1}
                &\Z/p^{k} \ar@/_.75pc/[d]_(.5){p^{m}}
                    &\Z \ar[l]\ar@/_.75pc/[d]_(.5){p^{m}}
                        &\Z \ar[l]_(.5){p^{k}}\ar@/_.75pc/[d]_(.5){p^{m}}
\\
{\mM (G/C_{p^{k+j}})}
    &\Z \ar[r]^(.5){p^{k}}\ar@/_.75pc/[u]_(.5){p^{m}}
                          \ar@/_.75pc/[d]_(.5){p^{s'}}
        &\Z \ar[r]\ar@/_.75pc/[u]_(.5){p^{m}}\ar@/_.75pc/[d]_(.5){1}
            &\Z/p^{k} \ar@/_.75pc/[u]_(.5){p^{m}}\ar@/_.75pc/[d]_(.5){1}
                &\Z/p^{k} \ar@/_.75pc/[u]_(.5){1}\ar@/_.75pc/[d]_(.5){1}
                    &\Z \ar[l]\ar@/_.75pc/[u]_(.5){1}
                                     \ar@/_.75pc/[d]_(.5){1}
                        &\Z \ar[l]_(.5){p^{k}}\ar@/_.75pc/[u]_(.5){1}
                                              \ar@/_.75pc/[d]_(.5){p^{s'}}
\\
{\mM (G/C_{p^{j+s}})}
    &\Z \ar[r]^(.5){p^{s}}\ar@/_.75pc/[u]_(.5){1}\ar@/_.75pc/[d]_(.5){p^{s}}
        &\Z \ar[r]\ar@/_.75pc/[u]_(.5){p^{s'}}\ar@/_.75pc/[d]_(.5){1}
            &\Z/p^{s} \ar@/_.75pc/[u]_(.5){p^{s'}}\ar@/_.75pc/[d]_(.5){}
                &\Z/p^{s} \ar@/_.75pc/[u]_(.5){p^{s'}}\ar@/_.75pc/[d]_(.5){}
                    &\Z \ar[l]\ar@/_.75pc/[u]_(.5){p^{s'}}
                                     \ar@/_.75pc/[d]_(.5){1}
                        &\Z \ar[l]_(.5){p^{s}}\ar@/_.75pc/[u]_(.5){1}
                                          \ar@/_.75pc/[d]_(.5){p^{s}}
\\
{\mM (G/C_{p^{j}})}
    &\Z \ar[r]^(.5){1}\ar@/_.75pc/[u]_(.5){1}\ar@/_.75pc/[d]_(.5){p^{j}}
        &\Z \ar[r]\ar@/_.75pc/[u]_(.5){p^{s}}
                         \ar@/_.75pc/[d]_(.5){p^{j}}
            &0 \ar@/_.75pc/[u]_(.5){}\ar@/_.75pc/[d]_(.5){}
                &0 \ar@/_.75pc/[u]_(.5){}\ar@/_.75pc/[d]_(.5){}
                    &\Z \ar[l]\ar@/_.75pc/[u]_(.5){p^{s}}
                                     \ar@/_.75pc/[d]_(.5){1}
                        &\Z \ar[l]_(.5){1}\ar@/_.75pc/[u]_(.5){1}
                                          \ar@/_.75pc/[d]_(.5){1}
\\
{\mM (G/e)}
    &\Z \ar[r]^(.5){1}\ar@/_.75pc/[u]_(.5){1}
        &\Z \ar[r]\ar@/_.75pc/[u]_(.5){1}
            &0 \ar@/_.75pc/[u]_(.5){}
                &0\ar@/_.75pc/[u]_(.5){}
                    &\Z \ar[l]\ar@/_.75pc/[u]_(.5){p^{j}}
                        &\Z \ar[l]_(.5){1}\ar@/_.75pc/[u]_(.5){p^{j}}
}
\end{displaymath}

We will also need the following hybrid of $\mB_{k,j}$ and $\mB_{k,j}^{*}$.

\begin{defn}
For $0\leq \ell$, let $\mB_{k,j}^{\ell}$ be the Mackey functor which agrees with
$\mB_{k,j}$ when restricted to $\Cp{\ell+k+j}$ and which for groups
containing $\Cp{\ell+k+j}$, agrees with $\mB_{k,j}^{\ast}$.

For $-k\leq \ell\leq 0$, let $\mB_{k,j}^{\ell}$ be $\mB_{k+\ell,j}$.
\end{defn}

In other words, $\mB_{k,j}^{\ell}$ becomes the dual once we reach
$C_{p^{n}}/\Cp{\ell+k+j}$ while moving up the Lewis diagram. In particular,
\begin{displaymath}
\mB_{k,j}^{0}= \mB_{k,j}^{*} 
\qquad \mbox{and}\qquad  
\mB_{k,j}^{m}= \mB_{k,j},
\end{displaymath}

\noindent where again $m=n-k-j$. For $0<\ell <m$, $\mB_{k,j}^{\ell
}$ is the quotient of a map between $\Z$-valued Mackey functors, but
not between the ones defined above.  Here is another partial Lewis
diagram illustrating this definition.

\begin{displaymath}
\xymatrix
@R=3mm
@!C=20mm
{{\mM }
    &{\mB_{k,j}=\mB_{k,j}^{m}}
        &{\mB_{k,j}^{\ell }}
            &{\mB_{k,j}^{*}=\mB_{k,j}^{0}}\\
{\mM (G/G)}
    &\Z/p^{k} \ar@/_1.5pc/[d]_(.5){1}
        &\Z/p^{k} \ar@/_1.5pc/[d]_(.5){p^{m-\ell }}
            &\Z/p^{k} \ar@/_1.5pc/[d]_(.5){p^{m-\ell }}\\
{\mM (G/C_{p^{\ell +k+j}})}
    &\Z/p^{k} \ar@/_1.5pc/[u]_(.5){p^{m-\ell }}\ar@/_1.5pc/[d]_(.5){1}
        &\Z/p^{k} \ar@/_1.5pc/[u]_(.5){p^{\ell }}\ar@/_1.5pc/[d]_(.5){1}
            &\Z/p^{k} \ar@/_1.5pc/[u]_(.5){1}\ar@/_1.5pc/[d]_(.5){p^{\ell }}\\
{\mM (G/C_{p^{k+j}})}
    &\Z/p^{k} \ar@/_1.5pc/[u]_(.5){p^{\ell }}\ar@/_1.5pc/[d]_(.5){1}
        &\Z/p^{k} \ar@/_1.5pc/[u]_(.5){1}\ar@/_1.5pc/[d]_(.5){1}
            &\Z/p^{k} \ar@/_1.5pc/[u]_(.5){1}\ar@/_1.5pc/[d]_(.5){1}\\
{\mM (G/C_{p^{j+s}})}
    &\Z/p^{s} \ar@/_1.5pc/[u]_(.5){p^{k-s}}\ar@/_1.5pc/[d]_(.5){}
        &\Z/p^{s} \ar@/_1.5pc/[u]_(.5){p^{k-s}}\ar@/_1.5pc/[d]_(.5){}
            &\Z/p^{s} \ar@/_1.5pc/[u]_(.5){p^{k-s}}\ar@/_1.5pc/[d]_(.5){}\\
{\mM (G/C_{p^{j}})}
    &0 \ar@/_1.5pc/[u]_(.5){}\ar@/_1.5pc/[d]_(.5){}
        &0 \ar@/_1.5pc/[u]_(.5){}\ar@/_1.5pc/[d]_(.5){}
            &0 \ar@/_1.5pc/[u]_(.5){}\ar@/_1.5pc/[d]_(.5){}\\
{\mM (G/e)}
   &0 \ar@/_1.5pc/[u]_(.5){}
        &0 \ar@/_1.5pc/[u]_(.5){}
            &0\ar@/_1.5pc/[u]_(.5){}
}
\end{displaymath}

In our analysis of $S^{k\lambda}\wedge H\mZ$, we will only need those
$\mB_{k,j}$ with $j=0$. To simplify notation, let $\mB_{k}$ be the
Mackey functor $\mB_{k,0}$.

\section{The \texorpdfstring{$RO(\Cp{n})$}{RO(G)}-graded homotopy of
 \texorpdfstring{$H\mZ$}{HZ}}

\subsection{Realizing forms of \texorpdfstring{$\mZ$}{Z}  topologically}
A nice feature of the Mackey functor $\mZ(k,j)$ is
that $H\mZ(k,j)$ is a virtual representation sphere smashed with
$H\mZ$. We will see this via a direct chain complex computation.

\begin{thm}\label{thm:FormofZTop}
For $j<k$, we have an equivalence
\[
\Sigma^{\lambda_{k}-\lambda_{j}}H\mZ\simeq H\mZ(k,j).
\]
\end{thm}

\begin{proof}
We use the simplified cell structures described above. If we smash
$S^{\lambda_{k}}$ with the cell structure for $S^{-\lambda_{j}}$, then
we have a complex of the form
\[
 \big({\Cpn} /{\Cp{j}}\big)_{+}\wedge 
          S^{0}\cup  \big({\Cpn} /{\Cp{j}}\big)_{+}\wedge 
             e^{1}\cup e^{0}\wedge S^{\lambda_{k}}.
\]
Smashing with $H\mZ$ and applying $\underline{\pi_{\ast}}$ gives a
chain complex of Mackey functors

\[
\xymatrix
@R=5mm
@C=30mm
{
{\m{\Z[{\Cpn} /{\Cp{k}} ]}}\ar[d]_{1-\gamma}
         \ar[dr]^{\Ind_{\cp{k}}^{\cp{n}}\res_{\cp{j}}^{\cp{k}}} 
    & {} \\
{\m{\Z[{\Cpn} /{\Cp{k}} ]}}
         \ar[d]\ar[dr]_{-{\Ind_{\cp{k}}^{\cp{n}}\res_{\cp{j}}^{\cp{k}}}} 
    & {\m{\Z[{\Cpn} /{\Cp{j}} ]}}\ar[d]^{1-\gamma} \\
{\mZ} 
    & {\m{\Z[{\Cpn} /{\Cp{j}} ]}}}
\]
The map labeled ${\Ind_{\cp{k}}^{\cp{n}}\res_{\cp{j}}^{\cp{k}}}$ is
actually the canonical inclusion of fixed points for the
subgroups. Thus there is no homology in the top degree. Similarly, the
kernel of the left most $1-\gamma$ is $\mZ$, which maps isomorphically to
the kernel of the second $1-\gamma$. Thus the homology of our chain
complex is the homology of the simpler complex:
\[
\xymatrix
@R=5mm
@C=30mm
{
{\mZ(n,k)}\ar[d]\ar[dr]^{-{\Ind_{\cp{k}}^{\cp{n}}\res_{\cp{j}}^{\cp{k}}}} 
      & {}\\
{\mZ} & {\mZ(n,j)}}.
\]
Now we make a few observations. For subgroups of ${\Cp{k}} $, the map 
\[
\mZ(n,k)\to\mZ
\]
is an isomorphism, and therefore the pushout looks like
$\mZ(n,j)$. For subgroups of ${\Cpn} $ which contain ${\Cp{k}} $, the
map
\[
\mZ(n,k)\to\mZ(n,j)
\]
is an equivalence, and therefore the pushout looks like $\mZ$. This is
the description of $\mZ(k,j)$.
\end{proof}
 
\begin{cor}
As a Mackey functor,
\[
\m{\pi}_{\lambda_{k}}S^{\lambda_{j}}\wedge H\mZ=\mZ(k,j).
\]
\end{cor}

Using the multiplicative structure of $H\mZ$, we can extend this to a
map of $H\mZ$-module spectra. Let
\[
v_{j/k}\colon S^{\lambda_{k}}\wedge H\mZ\to S^{\lambda_{j}}\wedge H\mZ
\]
be a generator of $\mZ(k,j)(\Cpn/\Cpn)$ which restricts to a chosen
orientation of the underlying $2$-spheres.

Since we are working at an odd prime, all representations of $\Cpn$
are orientable. In the equivariant context, this means that
\[
\m{\pi}_{\dim V} (S^{V}\wedge H\mZ)\cong \m{H}_{\dim V}(S^{V};\mZ)\cong\mZ.
\]
A choice of generator for $\mZ(\Cpn/\Cpn)$ determines one for all
other values of the Mackey functor and, since all restriction and
transfer maps are injective, this is detected in the underlying
homotopy. Let
\[
u_{V}\colon S^{\dim V}\to S^{V}\wedge H\mZ
\]
be such a choice of generator. We will refer to these as ``orientation
classes''. By choosing one for each irreducible representation of
$\Cpn$, we can ensure this is compatible with the multiplicative
structure in the sense that
\[
u_{V\oplus W}=u_{V}\cdot u_{W}.
\]

\begin{prop}
The composite
\[
S^{2}\xrightarrow{u_{\lambda_{k}}}S^{\lambda_{k}}\wedge H\mZ
     \xrightarrow{v_{j/k}} S^{\lambda_{j}}\wedge H\mZ,
\]
is homotopic to $u_{\lambda_{j}}$. Thus
\[
v_{j/k}=\frac{u_{\lambda_{j}}}{u_{\lambda_{k}}}.
\]
\end{prop}
\begin{proof}
The composite is detected in $\m{\pi}_{2}S^{\lambda_{j}}\wedge
H\mZ\cong\mZ$, and it is determined by the underlying map. By
assumption, this is a map of degree $1$.
\end{proof}

\subsection{The  gold  (or $au$) relation} 
The orientation classes $u_{\lambda_{k}}$ carve out a portion of the
$RO(\Cpn)$-graded homotopy of $H\mZ$. They do not tell a complete
picture, however. We need some classes in the Hurewicz image.

\begin{defn}
For a representation $V$, let 
\[
a_{V}\colon S^{0}\to S^{V}
\]
denote the inclusion of the origin and the point at infinity into $S^{V}$.
\end{defn}

We will often refer to the classes $a_{V}$ as ``Euler classes''. It is
obvious that if $V^{G}\neq \{0\}$, then $a_{V}$ is null-homotopic (as
the map factors through the inclusion $S^{V^{G}}\to S^{V}$). It
follows that the restriction of $a_{V}$ to any subgroup for which
$V^{H}\neq\{0\}$ is also null-homotopic. For groups for which
$V^{H}=\{0\}$, the class $a_{V}$ and all of its suspensions and powers
are essential. We will also let $a_{V}$ denote the Hurewicz image. Our
chain complex models show that this generates $H_{0}(S^{V};\mZ)$.

The classes $a_{\lambda_{k}}$ can be generalized in an interesting
way. If $j<k<n$, there are unstable maps
\[
\frac{a_{\lambda_{k}}}{a_{\lambda_{j}}}\colon S^{\lambda_{j}} 
   \to S^{\lambda_{k}}.
\]
This can be formed most easily by observing that the $p^{k-j}$th power
map extends to the one point compactifications. 

Composing with the Euler class $a_{\lambda_{j}}$ obviously results
in $S^{\lambda_{k}}$, and by our assumptions on $k$, this map is
essential (and in fact, torsion free).

Smashing this with $H\mZ$ produces an essential map which we will also
call by the same name
\[
\frac{a_{\lambda_{k}}}{a_{\lambda_{j}}}\colon S^{\lambda_{j}}\wedge
H\mZ\to S^{\lambda_{k}}\wedge H\mZ.
\]

Our chain complex models show that the classes $u_{\lambda_{0}}$,
$\dots$, $u_{\lambda_{n-1}}$ and $a_{\lambda_{0}}$, $\dots$,
$a_{\lambda_{n-1}}$ generate as an algebra the portion of the
$RO(\Cpn)$-graded homotopy of $H\mZ$ in dimensions of the form
$k-V$. There is a relation which is immediate from our fractional
classes.

\begin{thm}
In the $RO(\Cpn)$-graded homotopy of $H\mZ$, we have a relation
\[
a_{\lambda_{k}}u_{\lambda_{j}}=p^{k-j}a_{\lambda_{j}}u_{\lambda_{k}}.
\]
\end{thm}
\begin{proof}
If we compose $a_{\lambda_{k}}/a_{\lambda_{j}}$ with the map
$u_{\lambda_{j}}/u_{\lambda_{k}}$, then we get a map
\[
\frac{u_{\lambda_{j}}}{u_{\lambda_{k}}}
\frac{a_{\lambda_{k}}}{a_{\lambda_{j}}}\colon S^{\lambda_{j}}\wedge
H\mZ\to S^{\lambda_{j}}\wedge H\mZ.
\]
This map is completely determined by the effect in underlying
homotopy, and this is obviously the map of degree $p^{k-j}$. Clearing
denominators gives the desired result.
\end{proof}

\begin{remark}
There is only one other kind of relation in this sector of the\linebreak 
$RO(\Cpn)$-graded homotopy:
\[
p^{n-j}a_{\lambda_{j}}=0.
\]
This follows from the above remarks, as the restriction to $\Cp{j}$ of
$a_{\lambda_{j}}$ is null-homotopic. Composing with the transfer
yields multiplication by $p^{n-j}$ in homology. The cellular models
for representation spheres show that no other relations are possible:
all homology groups
\[
H_{k-V}(S^{0};\mZ)
\]
are cyclic.
\end{remark}

\subsection{Fiber sequences realizing \texorpdfstring{$H\mB_{k,j}$}{HB}.}

If we apply the {\EM} functor to the short exact sequence
\[
\mZ(k+j,j)\to \mZ\to \mB_{k,j},
\]
then we get a fiber sequence of spectra. Combing this with equivalence
from Theorem~\ref{thm:FormofZTop} yields a fiber sequence
\[
\Sigma^{\lambda_{k}-\lambda_{j}}H\mZ\to H\mZ\to H\mB_{k,j}.
\]

Since $\mB_{k,j}$ is zero on cells induced up from subgroups of
${\Cp{j}} $, for all $i\leq j$, the map $a_{\lambda_{i}}$ induces an
equivalence
\[
H\mB_{k,j}\simeq \Sigma^{\lambda_{i}}H\mB_{k,j}.
\]

Smashing our fiber sequence with $S^{\lambda_{j}}$ then
gives a more readily interpreted form.

\begin{prop}\label{prop:FiberSequence}
There is a fiber sequence
\[
\Sigma^{\lambda_{k}}H\mZ\to\Sigma^{\lambda_{j}}H\mZ\to H\mB_{k,j}.
\]
\end{prop}

Proposition~\ref{prop:FiberSequence} is the heart of our analysis of
the slice filtration for these\linebreak 
$RO(G)$-graded suspensions: we can strip away copies of
$S^{\lambda_{j}}$, replacing them with copies of $S^{\lambda_{k}}$ for
$k>j$.

Before continuing, we note two more natural fiber sequences which tie
in the dual Mackey functors.

\begin{prop}
We have fiber sequences
\[
\Sigma^{\lambda_k-\lambda_0}H\mZ\to
    \Sigma^{\lambda_j-\lambda_0}H\mZ\to
        \mB_{k,j}
\]
and
\[
\Sigma^{\lambda_n-\lambda_j}H\mZ\to
    \Sigma^{\lambda_n-\lambda_k}H\mZ\to
        \mB_{k,j}^\ast.
\]
\end{prop}

\section{Bredon homology computations}

We will see in the \S\ref{sec-slices} below that each slice for
$s^{\infty \lambda }\wedge H\mZ$ has the smash product of a certain
{\rep} sphere with some $H\mB_{k}$.  The relevant homotopy groups will
be given in Theorem \ref{thm:BredonHomology}. \S\ref{subsec-4.1} gives some
preparatory calculations for this result.  \S\ref{subsec-4.2} gives
some results that will help us determine the slices in
\S\ref{sec-slices}.

\subsection{Bredon homology and \texorpdfstring{$\mB_{k,j}$}{Bkj}}
\label{subsec-4.1}

\subsubsection{Exact Sequences involving $\mB$} 

A small initial observation is needed. The Mackey functors $B_{(k,j}$
do not really depend on $\Cp{n}$ in the sense that once $n$ is at
least $(k+j)$, the value for the Mackey functors on orbits
stabilizes. Thus in the short exact sequences that follow, it is
helpful to pretend they are written as Mackey functors for the
Pr\"ufer group $\Cp{\infty}$ and then restrict the result to any
chosen $\Cp{n}$. The net result is that if any index is greater than
$n$, just set it equal to $n$ in the formulas.

To compute the $E_{2}$-page of the slice spectral sequence for $S^{m\lambda }$,
we need to determine $H_{\ast}(S^{V};\mB_{k,j})$ for various
$V$. For this, we make several additional observations. If we let
$\phi_{p^{j}}$ denote the quotient map $\Cp{n}\to\Cp{n}/\Cp{j}$, then
there is an associated inflation (aka pullback) functor
\[
\phi_{p^{j}}^{\ast}\colon \Cp{n}/\Cp{j}\mhyphen\mathcal{M}ackey
\longrightarrow\Cp{n}\mhyphen\mathcal{M}ackey.
\]

The Mackey functors $\mB_{k,j}$ are all in the
image of $\phi_{p^{j}}^{\ast}$, and this is an exact functor. For us,
it simply inserts zeros at the bottom of a Lewis diagram.

\begin{prop}
We have natural isomorphisms
\[
\phi_{p^{j}}^{\ast} \mB_{k,m}
\cong \mB_{k,m+j}\text{ and }
\phi_{p^{j}}^{\ast} \mB_{k,m}^{\ell}
\cong \mB_{k,m+j}^{\ell}.
\]
\end{prop}
It therefore suffices to consider the case of $j=0$. We then have the
following generalization of Proposition~\ref{prop-2.3}, and the proof
is a direct computation.

\begin{prop}\label{prop:InducedBExact}
We have an exact sequence
\[
0 \to\mB_{\min(\ell,k),0} 
  \to\Ind_{\cp{\ell}}^{\cp{n}}\Res_{\cp{\ell}}^{\cp{n}} \mB_{k,0}
  \xrightarrow{\gamma-1}\Ind_{\cp{\ell}}^{\cp{n}}
                        \Res_{\cp{\ell}}^{\cp{n}} \mB_{k,0}
  \to\mB_{k,0}^{\ell-k}
   \to 0.
\]
\end{prop}

Proposition~\ref{prop:InducedBExact} is the initial input for
computing the Mackey functor cellular homology. The map $\gamma-1$ is
the attaching map of the even cells in $S^{V}$ to the odd cells when
we use our minimal cell model. Thus the cellular homology is
completely determined by looking at the maps from the cokernels in
Proposition~\ref{prop:InducedBExact} to the kernels as $\ell$
varies. We have two intermediate lemmata which serve as the work-horse
for any computations. These are immediate computations, and in both
cases, the middle map is an isomorphism when we restrict to the
subgroup for which both middle Mackey functors coincide.
 
\begin{lemma}\label{lem:ShortExactSequence}
For $\ell\geq -k$, we have an exact sequence
\[
0  \to \mB_{\min(k,k+\ell),\max(k,k+\ell)}^{\ast}
   \to \mB_{k,0}^{\ell}
   \to \mB_{k,0}
   \to \mB_{k,k+\ell}
   \to 0
\]
in which the middle map is the identity when restricted to $C_{p^{k+\ell }}$.
\end{lemma}

Pulling this back gives what happens for a general pair $(k,j)$:
\begin{cor}
For any $j$ and for $\ell\geq -k$, we have an exact sequence
\[
0  \to \mB_{\min(k,k+\ell),\max(k,k+\ell)+j}^{\ast}
   \to \mB_{k,j}^{\ell}
   \to \mB_{k,j}
   \to \mB_{k,k+\ell+j}
   \to 0
\]
in which the middle map is the identity when restricted to  $C_{p^{k+\ell+j }}$.
\end{cor}
 
\subsubsection{The Bredon homology of $S^{V}$ with coefficients in $\mB_{k,j}$}

We can splice all of the exact sequences from the previous subsection
to determine the Bredon homology of any representation sphere with
coefficients in $\mB_{k,j}$. This section is only relevant for
readers who wish to run the slice spectral sequence, as it
explains how to completely determine the slice $E_{2}$-term for
$S^{\infty\lambda}\wedge H\mZ$ and for $S^{m\lambda}\wedge H\mZ$.

First, we reduce to the cases most of interest.
\begin{prop} \label{Prop:BredonDoesntSee}
Let $W$ be a representation of $\Cp{n}$, let $W_{j}=W^{\Cp{j+1}}$, and let $U_{j}$ be the orthogonal complement of $W_{j}$ in $W$. Then the natural inclusion
\[
a_{U_{j}}\colon S^{W_{j}}\to S^{W}
\]
induces an equivalence upon smashing with $H\mB_{k,j}$.
\end{prop}
 
\begin{proof}
By assumption, the stabilizer of every point of $U_{j}-\{0\}$ is
properly contained in $\Cp{j+1}$. We therefore conclude that we can
choose a cellular model for $S^{U_{j}}$ such that the zero skeleton is
$S^{0}$ and every cell of dimension greater than zero is induced from
a proper subgroup of $\Cp{j+1}$. Since the restriction of
$\mB_{k,j}$ to any proper subgroup of $\Cp{j+1}$ is zero, we
conclude that the Euler class $a_{U_{j}}$ is an equivalence upon
smashing with $H\mB_{k,j}$.
\end{proof} 

Thus the Bredon homology computation we care about cannot distinguish
between $S^{W}$ and its $\Cp{j+1}$-fixed points. It is also immediate
that the chain complex computing the Bredon homology is in the image
of the pullback functor, so without loss of generality, we may assume
that $j=0$.

It is also clear that we may, without loss of generality, restrict attention to fixed point free representations of $\Cp{n}$, since the inclusion of a trivial summand serves only to shift the homology groups around.

\begin{notation}\label{nota:RepandKs}
Let $V$ be a fixed-point free representation of $\Cp{n}$ that
restricts trivially to $C_{p}$. Then find a $JO$-equivalent
representation $W$, and write
\[
W=\sum_{0<r\leq n}k_{r}\lambda_{n-r},
\]

Define natural numbers $K_{i}$ for $i>0$ by
\[
K_{i}=2\sum_{r=1}^{i}k_{r},
\]
and let $K_{0}=0$.  

Finally, let $0<i_{0}<\dots<i_{m}$ denote those integers $i$ such that $k_{i}\neq 0$, and let $h_{r}=n-i_{r}$.
\end{notation}

Since the homology groups of any space are bound between its
connectivity and its dimension, we need only determine the homology
groups between dimensions $0$ and $K_{m}$.

\begin{theorem}\label{thm:BredonHomology}

The Mackey functor valued Bredon homology groups of $S^{V}$ with coefficients in $\mB_{k}$ are given by
\[
\m{H}_{s}(S^{V};\mB_{k})=
\begin{cases}
0    & s<0,\\
\mB_{k,h_{0}} 
     & s=0\\
\mB_{\min(h_{0},k),\max(h_{0},k)}^{\ast} 
     & s=1 \\ 
\mB_{\min(k,n-i),n-i} 
     & K_{i-1}+2\leq s\leq K_{i}-2, i\text{ even} \\
\mB_{\min{(k,n-i)},n-i}^{\ast} 
     & K_{i-1}+3\leq s\leq K_{i}-1, i\text{ odd}\\
\mB_{\min{(k,h_{r})}, h_{r+1}} 
     & s=K_{i_{r}} \\
\mB_{\min{(k,h_{r+1})},\max (\min(k,h_{r}),h_{r+1})} 
     & s=K_{i_{r}}+1.
\end{cases}
\]
\end{theorem}

\begin{proof}
This follows immediately from our cell structure of
Proposition~\ref{prop:CellStructure} and the short exact sequences
above.

Applying the Mackey functor $\mB_{k}$ to the cellular complex results
in a chain complex of Mackey functors:
\begin{displaymath}
\xymatrix
@R=5mm
@C=4mm
{
%
{\mB_{k}}
    &{\Ind_{{h}_{0}}\Res_{{h}_{0}}\mB_{k}}\ar[l]
        &{\Ind_{{h}_{0}}\Res_{{h}_{0}}\mB_{k}}\ar[l]_(.5){1-\gamma }
            &{\Ind_{{h}_{0}}\Res_{{h}_{0}}\mB_{k}}\ar[l]
                &{\Ind_{{h}_{0}}\Res_{{h}_{0}}\mB_{k}}\ar[l]_(.5){1-\gamma }
                    &   \dotsb \ar[l]
                        &{\Ind_{{h}_{m}}\Res_{{h}_{m}}\mB_{k}}\ar[l]
}
\end{displaymath}

(for ease of exposition, we have assumed that $k_{i_{0}}\geq 2$). By
our assumptions on $V$, we know that the sole instance of $\mB_{k}$
lies in dimension $0$, and the complex is homologically graded. In
particular, there is the map $1-\gamma$ from every positive even
degree to the adjacent odd degree, up through $\dim
V$. Proposition~\ref{prop:InducedBExact} allows us to identify the
kernels and cokernels of the maps labeled $1-\gamma$, letting us
replace the cellular chains with a simpler complex:
\[
\xymatrix
@R=10mm
@C=7mm
{
{\mB_{k}} 
    & {\Ind_{{h}_{0}}\Res_{{h}_{0}}\mB_{k}}\ar[l] \ar@{->>}[d] 
        & {\Ind_{{h}_{0}}\Res_{{h}_{0}}\mB_{k}}\ar[l]_{1-\gamma} 
            & {\Ind_{{h}_{0}}\Res_{{h}_{0}}\mB_{k}}\ar[l] \ar@{->>}[d] 
                & {\Ind_{{h}_{0}}\Res_{{h}_{0}}\mB_{k}}\ar[l]_{1-\gamma} 
                    & {\dots}\ar[l]  \\
{}  & {\mB_{k,0}^{{h}_{0}-k}}\ar@{.>}[ul] 
        & {\mB_{\min({h}_{0},k),0}}\ar@{^{(}->}[u] 
            & {\mB_{k,0}^{{h}_{0}-k}}\ar@{.>}[l] 
                & {\mB_{\min({h}_{0},k),0}}\ar@{^{(}->}[u] 
                    & {\dots}\ar@{.>}[l]}
\]
Having controlled the differential from even degrees to odd degrees,
it remains only to understand the maps from the cokernels of the
various $1-\gamma$s to the kernels. This is exactly what
Lemma~\ref{lem:ShortExactSequence} records. The result is then a
simple application of this lemma and counting.
\end{proof}

\subsection{A criterion for slice codimension} \label{subsec-4.2}
In determining slice
dimensions, we will need a fact about the Bredon cohomology for a few
representation spheres.

In what follows, let $\m{H}^{k}(X;\mM)$ denote the Mackey functor
valued Bredon cohomology of $X$ with coefficients in $\mM$. This can
be computed by choosing an equivariant cellular model for $X$ and
building the cellular co-chain complex which assigns to an $n$-cell
with stabilizer $H$ the Mackey functor
$\Ind_{H}^{G}\Res_{H}^{G}\mM$. The maps are induced by the cellular
attaching maps in the standard way. Equivalently, we can simply build
a cellular model for $DX$, the Spanier-Whitehead dual of $X$, and
compute the negative Bredon homology of $DX$. 

We need this only for $X$ a representation sphere and $\mM$ a
$\mZ$-module, and even there, we need very few specific computations.

\begin{lemma}\label{lem:UpperBound}
Let $V$ be a representation of ${\Cpn} $. If $\lambda_{m}$ is a
summand of $V$, then for every $\mM$ for which
$\res_{\cp{m}}^{\cp{\ell}}$ is an isomorphism for $\ell\geq m$,
\[
\m{H}^{0}(S^{V};\mM)=\m{H}^{1}(S^{V};\mM)=0.
\]
\end{lemma}
\begin{proof}
If $\lambda_{n}$ is a summand, then connectivity implies the
result. The stabilizers of the complement of the origin in
$\lambda_{j}$ grows as $j$ does, and our cell models show that the
bottom skeleton of $S^{V}$ is that of $S^{\lambda_{j}}$ for $j$
maximal amongst those summands that occur in $V$. Without loss of
generality, we may assume $j=m$, and we need only compute
$\m{H}^{\epsilon}(S^{-\lambda_{m}};\mM)$.

Our standard model for $D(S^{\lambda_{m}})=S^{-\lambda_{m}}$ is 
\[
({\Cpn} /{\Cp{m}} )_{+}\wedge S^{-2}\cup ({\Cpn} /{\Cp{m}} )_{+}\wedge e^{-1}
\cup e^{0}.
\]
This gives us the following cochain complex for the cohomology of
$S^{\lambda_{m}}$ with coefficients in any Mackey functor $\mM$:
\[
\xymatrix@R=.5em{
{\m{C}^{\ast}(S^{\lambda_{m}};\mM):} 
    &{\mM}\ar[rr]^-{\underline{\res}_{p^{m}}^{p^{n}}}  
        &   & {\Ind_{{\cp{m}}}^{{\cp{n}}}\Res_{\cp{m}}^{\cp{n}}(\mM)} 
                          \ar[rr]^{1-\gamma} 
                &   &{\Ind_{{\cp{m}}}^{{\cp{n}}}\Res_{\cp{m}}^{\cp{n}}(\mM),}\\
{\text{dim}} 
    & {0} 
        &   &-1 &   & {-2}
}
\]
and the first map is
\begin{enumerate}
\item the composite of the restriction map to ${\Cp{m}} $ and the
diagonal for subgroups of ${\Cpn} $ which contain ${\Cp{m}} $, and
\item the diagonal for subgroups of ${\Cp{m}} $.
\end{enumerate}

For subgroups of ${\Cp{m}} $ and for any $\mM$, the sequence is
obviously exact.  This is really a restatement of the fact that
$\lambda_{m}$ restricts to the trivial representation for ${\Cp{m}}$.

For the remaining subgroups, we observe that the kernel of the map
denoted $(1-\gamma)$ looks like the constant Mackey functor
$\mM({\Cpn} /{\Cp{m}} )$ (the value for subgroups of ${\Cp{m}} $ is of
course {\em{a priori}} different, but those are already
understood). Thus we have an exact sequence for the cohomology of
$S^{\lambda_{m}}$ with coefficients in $\mM$ evaluated on ${\Cpn}
/{\Cp{\ell}} $ for $\ell\geq m$:
\begin{displaymath}
\xymatrix
@R=5mm
@C=10mm
{
{\underline{H}^{0}(S^{\lambda_{m}};\mM)(\Cpn/\Cp{\ell})}
              \ar@{^{(}->}[r]^(.5){}
    &\mM({\Cpn} /{\Cp{\ell}} )\ar[d]^(.5){\res_{p^{m}}^{p^{\ell}}}\\
    &\mM({\Cpn} /{\Cp{m}} )\ar@{->>}[r]^(.5){}
            &{\underline{H}^{1}(S^{\lambda_{m}};\mM)(\Cpn/\Cp{\ell}). }
}
\end{displaymath}
\end{proof}

\begin{cor}\label{cor:Bkcohomology}
If $\lambda_{m}$ is a summand of $V$ for $m\geq k$, then
\[
H^{0}(S^{V};\mB_{k})=H^{1}(S^{V};\mB_{k})=0.
\]
\end{cor}

\section{The slices of \texorpdfstring{$S^{\infty\lambda}\wedge
H\mZ$}{Sinfinitylambda}} \label{sec-slices}

Let $L=S^{\infty\lambda}\wedge H\mZ$.  We will see that its slice
tower is determined by a sequence of representations. The cofiber
sequences in the previous section, together with a surprisingly simple
induction argument, show that the naturally occurring tower is the
slice tower.

From this section on, let $p>2$.

\subsection{A sequence of representations} The basic argument is that
we will strip away copies of $\lambda$ from $L$, replacing them with
other irreducible representations until we produce a copy of
$2\rho$. Since the slice tower commutes with $\rho$-suspensions, this
will give us an iterative, periodic approach.

The sequence of representations is curiously simple. 

\begin{defn}\label{def-Vm}
For all $j\geq 1$, let 
\[
V_{j}=\bigoplus_{m=1}^{j}\lambda(2m-1),
\]
and let $V_{0}=0$.
\end{defn}

The first thing to observe is that $V_{p^{n}}=2\rho$
(we run through a complete set of coset representatives).  Since
$\lambda(k+p^{n})=\lambda(k)$ as representations of ${\Cpn} $, we
therefore conclude that
\[
V_{j+p^{n}}=2\rho+V_{j}
\]
for all $j\geq p^{n}$.

We also have $V_{(p^{n}\pm 1)/2}=\rho\pm 1 $ and $V_{p^{n}-j}=2\rho
-V_{j}$ for $0<j<p^{n}$.

The following formula for $V_{j}$ may be useful.
\begin{thm}\label{thm-explicit}
{\bf Floor function formula for $V_{j}$.} Let our group be
$C_{p^{n}}$ for $p$ an odd prime.  For $\ell\geq 0$, let ${c_{\ell }=
(p^{\ell }-1)/2}$.  Then
\begin{align*}
V_{j}
 & = \sum_{0\leq \ell <n}\lambda_{\ell }
      \sum_{0\leq t\leq p-1\atop t\neq c_{1}}
      \floor*{\frac{j+p^{\ell +1}-1-c_{\ell }-p^{\ell }t}{p^{\ell +1}}}
+2\floor*{\frac{j+c_{n}}{p^{n}}}.
\end{align*}
\end{thm}

\proof By definition, $V_{0}=0$ and 
\begin{displaymath}
V_{j}=V_{j-1}+\lambda (2j-1)=V_{j-1}+\lambda_{v_{p} (2j-1)}
\end{displaymath}

We will illustrate with the case $p=3$, for which $c_{1}=1$.  The
$\ell =0$ term in our sum is
\begin{align*}
\lambda_{0}\sum_{0\leq t\leq 2\atop t\neq 1}
      \floor*{\frac{j+3-1-c_{0}-t}{3}}
 & = \lambda_{0}\left(\floor*{\frac{j+2}{3}}+ \floor*{\frac{j}{3}} \right)
\end{align*}

\noindent Increasing $j$ by 1 increases the coefficient of
$\lambda_{0}$ when $j$ is congruent to 1 or 3 mod 3, namely the times
when $v_{3} (2j-1)=0$ and we are adding a copy of $\lambda_{0}$ to
$V_{j-1}$.

Similarly the  $\ell =1$ term 
\begin{align*}
\lambda_{1}\sum_{0\leq t\leq 2\atop t\neq 1}
      \floor*{\frac{j+9-1-c_{1}-3t}{9}}
 & = \lambda_{1}\left(\floor*{\frac{j+7}{9}}+ \floor*{\frac{j+1}{9}} \right)
\end{align*}

\noindent Increasing $j$ by 1 increases the coefficient of
$\lambda_{1}$ when $j$ is congruent to 2 or 8 mod 9, namely the times
when $v_{3} (2j-1)=1$ and we are adding a copy of $\lambda_{1}$ to
$V_{j-1}$.

The same thing happens for larger $\ell $ and larger primes.  

When we get to $\ell =n$, we have $\lambda_{n}=2$ by definition.  It
will be added to $V_{j-1}$ when ever $2j-1$ is divisible by $p^{n}$,
which is equivalent to $j$ being congruent to $-c_{n}$ modulo $p^{n}$.
The final coefficient, $\floor*{(j+c_{n})/p^{n}}$ increases for
precisely such $j$.  \qed\bigskip

\begin{cor}\label{cor-coeffs}
{\bf Coefficients in cases of interest.}

Let $j= (ap-1)/2$ for odd $a>0$, and let $a=2b+1$.  Then the coeffcient
$k_{n-\ell }$ of $\lambda_{\ell }$ in the formula of Theorem
\ref{thm-explicit} is
\begin{displaymath}
k_{n-\ell }=\mycases{
(p-1)b+c_{1}
       &\mbox{for }\ell =0\\
\displaystyle{\sum _{0\leq t\leq p-1\atop t\neq c_{1}}}
         \floor*{\dfrac{b +p^{\ell}-1-c_{\ell-1 }-p^{\ell -1} t}{p^{\ell}}}
       &\mbox{for }0<\ell <n\\
\floor*{\dfrac{bp+c_{1}+c_{n}}{p^{n}}}
       &\mbox{for }\ell =n.
}
\end{displaymath}

\noindent In the case $0<\ell <n$ we have  
\begin{displaymath}
(p-1)\floor*{\frac{b}{p^{\ell}}}\leq k_{n-\ell }
           \leq (p-1)\floor*{\frac{b+p^{\ell }}{p^{\ell}}}.
\end{displaymath}
\end{cor}

Note that $k_{0}$ is the coefficient of $\lambda_{n}$, which by
definition is the trivial representation of degree two,

\proof
The coefficient $k_{n}$ of $\lambda_{0}$ for $j= (ap-1)/2$ with $a$ odd is 
\begin{align*}
k_{n}
 & = \sum _{0\leq t\leq p-1\atop t\neq c_{1}}
         \floor*{\frac{j+p-1-t}{p}}   \\
 & = \sum _{0\leq t\leq p-1\atop t\neq c_{1}}
         \floor*{\frac{ap-1+2p-2-2t}{2p}}   \\
 & = \sum _{0\leq t\leq p-1\atop t\neq c_{1}}\left(\floor*{\frac{a-1}{2}}
         +\floor*{\frac{3p-3-2t}{2p}}  \right)  \\
 & = \sum _{0\leq t\leq p-1\atop t\neq c_{1}}\left(b+1
         +\floor*{\frac{p-3-2t}{2p}}  \right)  \\
 & = (p-1) (b+1)-c_{1} = (p-1)b+c_{1}= c_{1}a.
\end{align*}

\noindent That of $\lambda_{\ell}$ for $0<\ell <n$ is 
\begin{align*}
k_{n-\ell }
 & = \sum _{0\leq t\leq p-1\atop t\neq c_{1}}
         \floor*{\frac{(ap-1)/2 +p^{\ell +1}-1-c_{\ell }-p^{\ell }t}
                      {p^{\ell +1}}}   \\
 & = \sum _{0\leq t\leq p-1\atop t\neq c_{1}}
         \floor*{\frac{b p+c_{1}+p^{\ell+1}-1-c_{\ell }-p^{\ell }t}
                       {p^{\ell +1}}}   \\
 & = \sum _{0\leq t\leq p-1\atop t\neq c_{1}}
         \floor*{\frac{b +p^{\ell}-1-c_{\ell-1 }-p^{\ell -1} t}{p^{\ell}}}
\qquad \mbox{since $c_{1}-c_{\ell }=-pc_{\ell -1}$.} 
\end{align*}

\noindent Since 
\begin{displaymath}
0<1+c_{\ell-1 }+p^{\ell -1} t<p^{\ell }\qquad \mbox{for each $t$, } 
\end{displaymath}

\noindent we have 
\begin{displaymath}
(p-1)\floor*{\frac{b}{p^{\ell}}}\leq k_{n-\ell }
           \leq (p-1)\floor*{\frac{b+p^{\ell }}{p^{\ell}}}.
\end{displaymath}

We leave the case of $k_{0}$ to the reader. 
\qed\bigskip 

This will let us simplify several arguments. 

\subsection{Special slices} We are now in a position where we can show
that all of the $RO(G)$-graded suspensions of $H\mB_j$ which arise are
in fact slices. We begin with several simple theorems about
$V_{j}$-fold suspensions.

\begin{thm}\label{thm:Vjgeq2j}
For any $Y$ which is $(-1)$-connected, $S^{V_{j}}\wedge Y$ is slice
greater than or equal to $(2j)$.
\end{thm} 

\begin{proof}
Since $(-1)$-connected spectra are slice non-negative and since
smashing with a slice non-negative spectrum at worst preserves slice
connectivity, it suffices to show that $S^{V_{j}}$ is slice greater
than or equal to $(2j)$. Here we can invoke
Proposition~\ref{prop:SliceConnectivity}. We first observe that the
restriction of $V_{j}$ to any subgroup of $\Cpn$ is the corresponding
$V_{j}$ for that subgroup. Thus we may use induction on the order of
the group in a very simple away. For $0\leq j<p^{n}$ let
\[
W_{j}=\begin{cases} 
\rho-1  & j\leq \dfrac{p^{n}-1}{2} \\
3\rho-1 & j\geq \dfrac{p^{n}+1}{2},
\end{cases}
\]
then $V_{j}\subset W_{j}$ and $V_{j}^{G}=W_{j}^{G}$. By induction, for
all proper subgroups $H$, $i_{H}^{\ast}S^{V_{j}}$ is slice greater
than or equal to $(2j)$, and therefore Proposition~\ref{prop:SliceConnectivity}
implies that desired result.
\end{proof}

\begin{cor}
For any Mackey functor $\mM$, $S^{V_{j}}\wedge H\mM$ is slice greater
than or equal to $(2j)$.
\end{cor}

As is often the case, showing that the desired spectra are slice less
than or equal to $(2j)$ is by direct computation. We will reduce the
computation to Corollary~\ref{cor:Bkcohomology}.

\begin{thm}\label{thm-5.4}
For every $\mM$ in which $\res_{\cp{m}}^{\cp{\ell}}$ is an isomorphism
for\linebreak $m\geq k:=v_{p}(2j+1)$, $S^{V_{j}}\wedge H\mM$ is slice
less than or equal to $(2j)$.  Here $v_{p}(2j+1)$ denotes the number
of powers of $p$ dividing $2j+1$.
\end{thm}

\begin{proof}
We need to show that for all triples $(r,H,\epsilon)$ such that
$r|H|-\epsilon > 2j$, we have
\[
[G_{+}\wedge_{H}S^{r\rho_{H}-\epsilon},\Sigma^{V_{j}}H\mM]^{G}
    =[S^{r\rho_{H}-\epsilon},\Sigma^{V_{j}}H\mM]^{H}=0.
\]

We can again appeal to the linear ordering of the subgroups, breaking
them into two cases. First, observe that if $2j+1\equiv 0\mod p^{k}$,
then $j\equiv \tfrac{p^{k}-1}{2} \mod p^{k}.$ This means that
\[
i_{C_{p^{k}}}^{\ast} V_{j}=b\rho+\bar{\rho},
\]
for some $b$. Hence for any Mackey functor $\mM$, 
\[
i_{C_{p^{k}}}^{\ast}\Sigma^{V_{j}}H\mM
      \cong\Sigma^{b\rho+\bar{\rho}} Hi_{C_{p^{k}}}^{\ast}\mM
\]
is a $2j$-slice for $C_{p^{k}}$. This is the essential feature of the
argument, as it shows that the slice upper bound on this spectrum is
determined by the slice upper bound for those subgroups of $G$ which
properly contain $C_{p^{k}}$. The condition that all restriction maps are
injections is the same for all subgroups in this range, so without
loss of generality, we need only show
\[
[S^{r\rho_{G}-\epsilon},\Sigma^{V_{j}}H\mM]^{G}=0.
\]

{\em For the rest of this proof we will denote $\rho_{G}$ by simply
$\rho $.}  It suffices to show this for $j\leq p^{n}-1$, as larger
values of $j$ result in $\rho$-fold suspensions and these commute with
slice dimension. We therefore only have to consider slice cells
$S^{r\rho_{G}-\epsilon}$, where $rp^{n}-\epsilon > 2j$.

There are two cases, depending on whether or not we have passed the
special value of $j$: $\tfrac{p^{n}-1}{2}$.

For $1\leq j\leq \tfrac{p^{n}-1}{2}$, we know that $V_{j}\subset
S^{\bar{\rho}}$, with equality iff $j=\tfrac{p^{n}-1}{2}$. We
therefore must compute
\[
[S^{r\rho-\epsilon},\Sigma^{V_{j}}H\mM]
    =[S^{(r-1)\bar{\rho}+V_{j}^{\perp}+(r-\epsilon)},H\mM],
\]
where $V_{j}^{\perp}$ is the orthogonal complement of $V_{j}$ in
$\bar{\rho}$. If $r-\epsilon>0$, then the connectivity of the domain
exceeds the coconnectivity of the range, and therefore all homotopy
classes are zero. We pause here to note that this includes the
exceptional value of $j$, as here $r=\epsilon=1$ results in
$S^{r\rho-\epsilon}$ being a $2j$-slice cell.

For the remaining cases, we assume that $r=\epsilon=1$ and
$j<\tfrac{p^{n}-1}{2}$. By assumption on $j$, now, the representation
$\lambda(2j+1)\subset V_{j}^{\perp}$. The corresponding sphere is
$JO$-equivalent to the sphere of $\lambda_{v_{p}(2j+1)}$. We are
therefore computing
\[
H^{0}\big(S^{V_{j}^{\perp}};\mM\big),
\] 
and by Lemma~\ref{lem:UpperBound}, this group is zero.

For $\tfrac{p^{n}-1}{2}<j<p^{n}$, we have a similar analysis. Here
$\rho\subset V_{j}\subset 2\rho$. Assume that $rp^{n}-\epsilon >2j$,
as before. In particular, we see that $r\geq 2$. We again consider
\[
[S^{r\rho-\epsilon},\Sigma^{V_{j}}H\mM]
  =[S^{(r-2)\bar{\rho}+V_{j}^{\perp}+(r-2)-\epsilon},H\mM],
\]
where $V_{j}^{\perp}$ is the orthogonal complement of $V_{j}$ in
$2\rho$. If $(r-2)-\epsilon$ is positive, then connectivity finishes
the proof. We need to consider $(r-2)-\epsilon$ being $0$ or $-1$ (the
latter only occurring for $j<p^{n}$), and this reduces the computation
to that of
\[
H^{\epsilon}(S^{V_{j}^{\perp}+(r-2)\bar{\rho}};\mM).
\]
If $r\geq 3$, then all $\lambda_{i}$ are summands of $\bar{\rho}$, and
Lemma~\ref{lem:UpperBound} immediately implies these groups are
zero. If $r=2$, then just as before, we observe that $V_{j}^{\perp}$
contains the representation $\lambda(2j+1)$, and so
Lemma~\ref{lem:UpperBound} again implies that these groups are zero.
 
Finally we observe that since $v_{p}(2p^{n}+1)=0$, the
conditions of the theorem require that all restrictions maps be
injections. This means that $H\mM$ is a zero slice, and hence
\[
\Sigma^{V_{j}}H\mM=\Sigma^{2\rho}H\mM
\]
is a $2p^{n}$-slice.
\end{proof}

Combining these theorems yields a number of slices.

\begin{corollary}\label{cor:BjSlices}
For all $j$, the spectrum $\Sigma^{V_{j}}H\mB_{v_{p}(2j+1)}$ is a
$(2j)$-slice.
\end{corollary}

\begin{corollary}\label{cor:ZSlices}
For all $j$, the spectra $\Sigma^{V_{j}} H\mZ$ and
$\Sigma^{V_{j}}H\mZ(k,\ell)$ for $k\leq v_{p} (2j+1)$ or $\ell \geq
v_{p} (2j+1)$ are $(2j)$-slices.
\end{corollary}

The condition on $k$ and $\ell $ above is that $v_{p} (2j+1)\notin
(\ell ,k)$, the open interval from $\ell $ to $k$.

\subsection{The \texorpdfstring{$2\rho$}{2rho}-periodic slice tower}

The results of the previous section actually determine for us the
slice tower. It is easiest to observe this via the $2\rho$-suspensions
that showed up. As a corollary to the previous section, we have a
tower
\[
\xymatrix
@R=5mm
@C=10mm 
{
{\Sigma^{2\rho}L}\ar[r] 
    & {\Sigma^{2\rho-\lambda}L }\ar[r]\ar[d] 
        & {\Sigma^{2\rho-\lambda}H\mB_{0}}\\
{}  & {\vdots} \ar[d] 
        & {} \\
{}  & {\Sigma^{V_{j}}L} \ar[d]\ar[r] 
        & {\Sigma^{V_{j}}H\mB_{v_{p}(2j+1)}} \\
{}  & {\vdots}\ar[d] 
        & {}\\
{}  & {L}\ar[r] 
        & {\Sigma^{0}H\mB_{0}}}.
\]
Recall that $H\mB_{0}$ is contractible, $V_{p^{n}}=2\rho $ and
$V_{p^{n}-1}=2\rho -\lambda $.  All of the layers on the right-hand
side of the tower are slices by Corollary~\ref{cor:BjSlices}. In
particular, they are all simultaneously less than or equal to and
greater than or equal to the dimension of the associated
representation sphere. The final one (the one in the upper right
corner) is then less than or equal to $(2p^{n}-2)$. Now the $2\rho$th
suspension of $L$ is greater than or equal
to $2p^{n}$, and we therefore conclude that all of the cofiber
sequences are exactly the cofiber sequences
\[
P_{n}(X)\to X\to P^{n-1}(X)
\]
for various spectra $X$ of the form $S^{V}\wedge
L$. Splicing them all together, using the
obvious $2\rho$-periodicity of the tower, we see that we have
determined the slice co-tower of $L$. We
group this together in the following theorem.

\begin{thm}\label{thm-slices}
All odd slices and all slices in dimensions not congruent to $-1$
modulo $p$ of $L$ are contractible. The
$(ap^{k}-1)$-slice, where $a$ is odd and prime to $p$, is given by
\[
\Sigma^{V_{j}}H\mB_{k}\qquad \mbox{where }j=(ap^{k}-1)/2 , 
\]
and the maps in the tower are all determined by the cofiber sequences
\[
\Sigma^{V_{j+1}}H\mZ
    \xrightarrow{u_{\lambda}/u_{\lambda(2j+1)}}\Sigma^{V_{j}+\lambda}H\mZ
    \to\Sigma^{V_{j}}H\mB_{v_{p}(2j+1)}.
\]
\end{thm}

We leave it to the interested reader to state a corollary to Theorem
\ref{thm:BredonHomology} as it applies to the $\Sigma^{V_{j}}H\mB_{k}$
here.  In particular it would say that the top dimension for the
homology of the $ap^{k}-1$-slice is $ap^{k-1}-1$.  This means that in
the usual slice spectral sequence chart all nontrivial elements occur
between lines of slopes $p-1$ and $p^{n}-1$ meeting at $(s,t)=
(-1,0)$.

\begin{remark}
Yarnall's thesis shows that the slice sections of $S^{n}\wedge H\mZ$
are all of the form $S^{V}\wedge H\mZ$. This result is of a 
different flavor. Our result here is that the slice
{\emph{connective covers}} $P_{n}L$ are all
representations spheres smashed with $H\mZ$. This is a curious and
confusing fact.
\end{remark}

We will now apply Theorem \ref{thm:BredonHomology} to determine the homotopy
groups of the slices described by Theorem \ref{thm-slices}.

Figure~\ref{fig:SSSC27} shows the spectral sequence for $C_{27}$,
subect to the following regrading convention.  Under the usual
convention, meaning the point $(s,t)$ shows the Mackey functor
$\underline{E}_{2}^{s,t+s}$, all nontrivial elements would lie between
lines of slopes 2 and 26 intersecting at $(s,t)= (0,-1)$, meaning
$(x,y)= (-1,0)$. In order to save space we rescale in such a way that
the vanishing lines have slope 0 and 8 and the horizontal coordinate
is unchanged.  In the general case this means
\[
\begin{bmatrix}
x \\ y
\end{bmatrix} =
\begin{bmatrix}
t-s \\
s-\frac{p-1}{p}(t+1)
\end{bmatrix}.
\]
This change rescales differentials as well, converting $d_{1+2pr}$ for
$r>0$ (the only ones that can occur dues to sparseness) to $d_{1+2r}$.
The figure makes use of Mackey functor symbols indicated in Table
\ref{tab-Mackey-functors}.  The dashed lines are non-trivial extensions
determined by hidden transfers.

{\em The reader may construct similar charts for other cyclic groups
using the information in Corollary \ref{cor-coeffs}.}

\begin{table}[h] \caption{Mackey functors appearing in Figures
\ref{fig:SSSC27} and \ref{fig:S8lambda}.}  \label{tab-Mackey-functors}

\begin{tabular}[]{|c|c|c|c|c|c|c|}
\hline 
$\mB_{1}=\mB_{1,0} $
    &$\mB_{1,1}$
        &$\mB_{1,1}^{*}$
            &$\mB_{1,2}=\mB_{1,2}^{*}$
\\
\hline 
$\bullet$
    &$\underline{\bullet}$
        &$\underline{\bullet}^{*}$
            &$\underline{\underline{\bullet}}
                  =\underline{\underline{\bullet}}^{*}$
\\
\hline 
$
\xymatrix
@R=4mm
@C=5mm
{
\Z/3 \ar@/_1.5pc/[d]_(.5){1}\\
\Z/3 \ar@/_1.5pc/[d]_(.5){1}\ar@/_1.5pc/[u]_(.5){0}\\
\Z/3 \ar@/_1.5pc/[d]_(.5){}\ar@/_1.5pc/[u]_(.5){0}\\
0\ar@/_1.5pc/[u]_(.5){}
} 
$
    &
$
\xymatrix
@R=4mm
@C=5mm
{
\Z/3 \ar@/_1.5pc/[d]_(.5){1}\\
\Z/3 \ar@/_1.5pc/[d]_(.5){}\ar@/_1.5pc/[u]_(.5){0}\\
0    \ar@/_1.5pc/[d]_(.5){}\ar@/_1.5pc/[u]_(.5){}\\
0\ar@/_1.5pc/[u]_(.5){}
} 
$
        &
$
\xymatrix
@R=4mm
@C=5mm
{
\Z/3 \ar@/_1.5pc/[d]_(.5){0}\\
\Z/3 \ar@/_1.5pc/[d]_(.5){}\ar@/_1.5pc/[u]_(.5){1}\\
0    \ar@/_1.5pc/[d]_(.5){}\ar@/_1.5pc/[u]_(.5){}\\
0\ar@/_1.5pc/[u]_(.5){}
} 
$
            &
$
\xymatrix
@R=4mm
@C=5mm
{
\Z/3 \ar@/_1.5pc/[d]_(.5){}\\
0 \ar@/_1.5pc/[d]_(.5){}\ar@/_1.5pc/[u]_(.5){}\\
0    \ar@/_1.5pc/[d]_(.5){}\ar@/_1.5pc/[u]_(.5){}\\
0\ar@/_1.5pc/[u]_(.5){}
} 
$
\\
\hline \hline
$\mB_{2}= \mB_{2,0}$
    &$\mB_{2,1}$
        &$\mB_{3}= \mB_{3,0}$
            &$\mZ $\\
\hline 
$\circ $
    &$\underline{\circ }$
        &$\circledcirc $
            &$\Box$\\
\hline 

$
\xymatrix
@R=4mm
@C=5mm
{
\Z/9 \ar@/_1.5pc/[d]_(.5){1}\\
\Z/9 \ar@/_1.5pc/[d]_(.5){1}\ar@/_1.5pc/[u]_(.5){3}\\
\Z/3 \ar@/_1.5pc/[d]_(.5){}\ar@/_1.5pc/[u]_(.5){3}\\
0\ar@/_1.5pc/[u]_(.5){}
} 
$
    & 
$
\xymatrix
@R=4mm
@C=5mm
{
\Z/9 \ar@/_1.5pc/[d]_(.5){1}\\
\Z/3 \ar@/_1.5pc/[d]_(.5){}\ar@/_1.5pc/[u]_(.5){3}\\
0 \ar@/_1.5pc/[d]_(.5){}\ar@/_1.5pc/[u]_(.5){}\\
0\ar@/_1.5pc/[u]_(.5){}
} 
$
        &
$
\xymatrix
@R=4mm
@C=5mm
{
\Z/27 \ar@/_1.5pc/[d]_(.5){1}\\
\Z/9 \ar@/_1.5pc/[d]_(.5){1}\ar@/_1.5pc/[u]_(.5){3}\\
\Z/3 \ar@/_1.5pc/[d]_(.5){}\ar@/_1.5pc/[u]_(.5){3}\\
0\ar@/_1.5pc/[u]_(.5){}
} 
$
            &
$
\xymatrix
@R=5mm
@C=5mm
{
\Z \ar@/_1.5pc/[d]_(.5){1}\\
\Z \ar@/_1.5pc/[d]_(.5){1}\ar@/_1.5pc/[u]_(.5){3}\\
\Z \ar@/_1.5pc/[d]_(.5){1}\ar@/_1.5pc/[u]_(.5){3}\\
\Z \ar@/_1.5pc/[u]_(.5){3}
} 
$\\
\hline 
\end{tabular}
\end{table}

\bigskip

\begin{figure}[ht] 
\centering
\includegraphics[width=\textwidth]{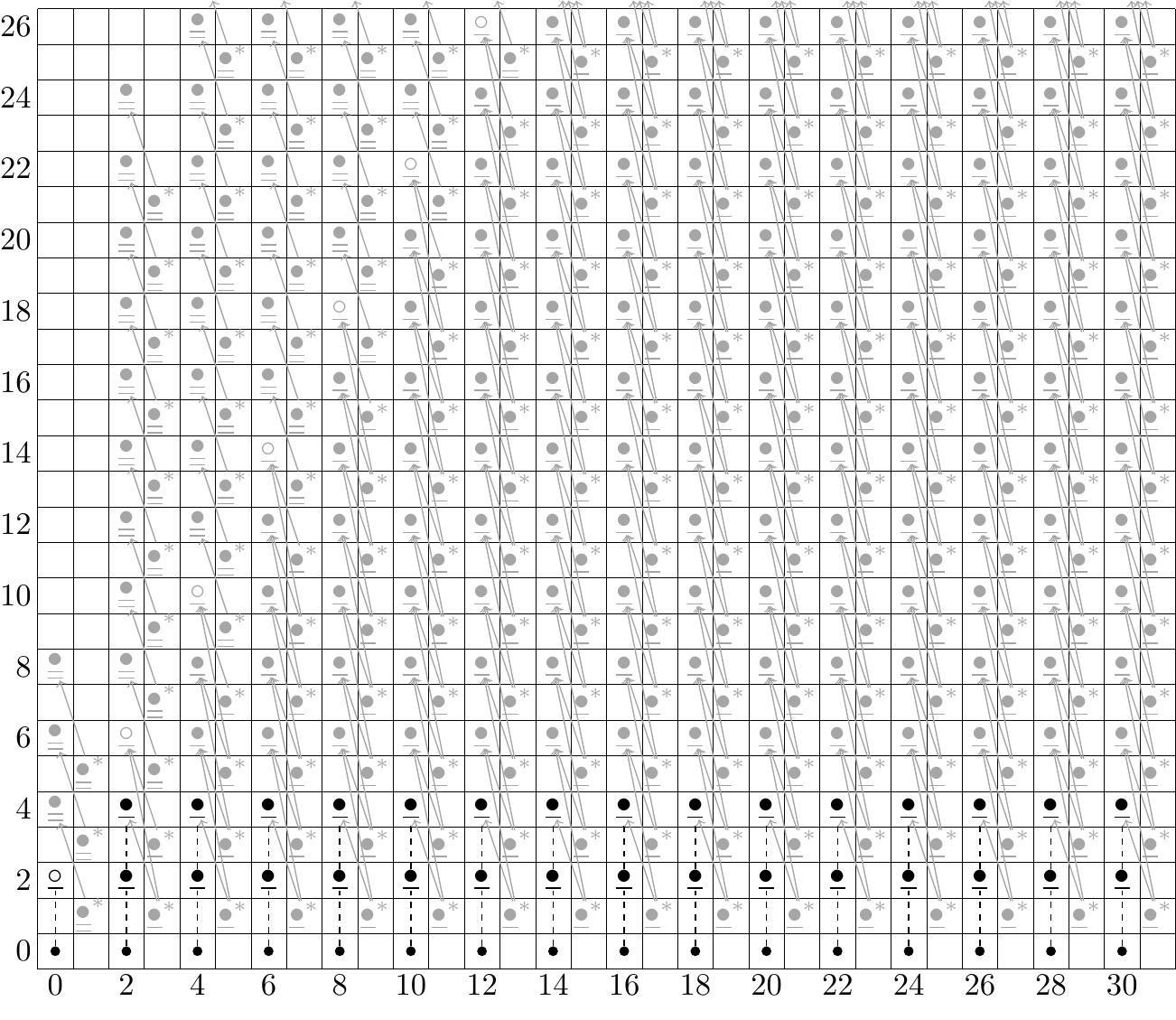} 

\caption{The slice spectral sequence for $S^{\infty\lambda}\wedge
H\mZ$ for $G=C_{27}$.} 
\label{fig:SSSC27}
\end{figure}

\section{The slices of \texorpdfstring{$S^{m\lambda}\wedge
H\mZ$}{finite lambda}}

The slice tower for $L= S^{\infty \lambda}\wedge H\mZ$ almost
completely determines the one for $S^{m\lambda}\wedge
H\mZ$. We first bound the slice tower of $S^{m\lambda}\wedge H\mZ$
from above.

\begin{thm}
For all $m\geq 0$, $S^{m\lambda}\wedge H\mZ$ is less than or equal to $2m$.
\end{thm}

\begin{proof}
Since $S^{m\lambda}\wedge H\mZ$ restricts to the same spectrum for any
subgroup, by induction on the order of $G$, we know that the
restriction is less than or equal to $2m$. This means that we need
only check that for all $k$ and $\epsilon$ such that $k\cdot
p^{n}-\epsilon>2m$,
\[
[S^{k\rho-\epsilon},S^{m\lambda}\wedge H\mZ]
    =[S^0,S^{m\lambda-k\rho+\epsilon}\wedge H\mZ]=0.
\]

We again can show this using a cell decomposition of $S^{m\lambda}$
and then smashing it with $S^{-k\rho+\epsilon}$. This gives
\begin{multline*}
\big(S^0\cup ( C_{p^n+}\wedge e^1)\cup\dots 
                       \cup ( C_{p^n+}\wedge e^{2m})\big)
          \wedge S^{-k\rho+\epsilon}\wedge H\mZ\\
= 
\big(S^{-k\rho+\epsilon}\cup ( C_{p^n+}\wedge e^{1-kp^n+\epsilon})\cup\dots
           \cup ( C_{p^n+}\wedge e^{2m-kp^n+\epsilon})\big)\wedge H\mZ.
\end{multline*}
The only ways this could have homotopy in dimension $0$ are possibly
from the first term (when $k=\epsilon=1$) or from the final term in the
decomposition. Since $2m-kp^n+\epsilon <0$ by assumption, all cells
except the first are $0$-coconnected. The standard computations show
that the first term also contributes no $\m{\pi}_0$.
\end{proof}

We therefore know we need only determine the slices up to the
$2m$\textsuperscript{th} slice. Our computation for
$L$ actually does most of this. Let $F_{m}$
be the fiber of the natural inclusion $S^{m\lambda}\wedge
H\mZ\hookrightarrow L$.

\begin{thm}
The spectrum $F_{m}$ is greater than or equal to $2m$.
\end{thm}

\begin{proof}
The long exact sequence in homotopy for the fiber sequence 
\[
F_{m}\to S^{m\lambda}\wedge H\mZ\to L
\]
shows that
\[
\m{\pi}_{s}(F_{m})=\begin{cases}
0 & s<2m,\\
\mZ & s=2m,\\
0 & s=2k>2m,\\
\mB(n,0) & s=2k+1>2m.
\end{cases}
\]

Since the spectrum $F_{m}$ is $(2m-1)$-connected, we know that it is
greater than or equal to $2m$.
\end{proof}

\begin{cor}
The natural map $S^{m\lambda}\wedge H\mZ\to S^{\infty\lambda}\wedge
H\mZ$ induces an equivalence
\[
P^{2m-1}\big(S^{m\lambda}\wedge H\mZ\big)
    \to P^{2m-1}\big(L\big).
\]
\end{cor}

We therefore know all of the slices of $S^{m\lambda}\wedge H\mZ$ below
the $2m$\textsuperscript{th} slice, and we also know that we have a
single remaining slice. Moreover, we know the maps down to the slice
sections, since they are the composites
\[
S^{m\lambda}\wedge H\mZ
      \to L
            \to P^{j}\big(L\big).
\]

We therefore have the slice tower. The fiber of the map
\[
S^{m\lambda}\wedge H\mZ\to P^{2m-1}\big(S^{m\lambda}\wedge H\mZ\big)
\]
is the $(2m)$-slice, since it is simultaneously greater than or equal
to $2m$ and less than or equal to $2m$. The analysis of the slice
tower for $L$ also identifies this with a representation sphere for
us, as determined by the analysis in the previous section.

Considering the actual fiber sequences for the slice cotower for
$L$ shows us what the $(2m)$-slice is for
$S^{m\lambda}\wedge H\mZ$. The determination of the map is immediate
from the consideration of the layers of the tower (and the description
of stripping off copies of $\lambda$ and replacing them with
$\lambda(2j-1)$ for $j$ at most $m$). We simplify the presentation by
taking advantage of the fact that $H\mZ$ is a commutative ring
spectrum.

\begin{thm}
For all $m\geq 0$, the $(2m)$-slice of $S^{m\lambda}\wedge H\mZ$ is
$S^{V_{m}}\wedge H\mZ$.

The map from $S^{V_{m}}\wedge H\mZ$ to $S^{m\lambda}\wedge H\mZ$ is simply
\[
\prod_{j=1}^{m}\frac{u_{\lambda(2j-1)}}{u_{\lambda}}
      =\frac{u_{V_{m}}}{u_{m\lambda}}
\]
\end{thm}

As a closing remark, we note a fact we found initially quite
curious. If $p$ is large relative to $m$, then $S^{m\lambda}\wedge
H\mZ$ is in fact a $(2m)$-slice. It takes a while before the slice
tower of $S^{m\lambda}\wedge H\mZ$ becomes more complicated.

To illustrate how this plays out, we show in Table~\ref{tab:mversusVm}
the $J$-equivalence classes of $V_{m}$ for $p=3$, $n=2$ and $1\leq m\leq 10$.

\begin{table}[ht]
 
\caption{$V_{j}$ for $C_{27}$ and $1\leq j\leq 27$. Note that
$V_{27-j}=2\rho -V_{j}$.} 
\label{tab:mversusVm}

\begin{tabular}[]{||c|c||c|c||c|c||}
\hline 
$j$ &$V_{j}$
        &$j$&$V_{j}$
                &$j$&$V_{j}$            \\
\hline 
1   &$\lambda $
        &2  &$\lambda +\lambda_{1}$
                &3  &$2\lambda +\lambda_{1}$\\
4   &$3\lambda+\lambda_{1} $
        &5  &$3\lambda +\lambda_{1}+\lambda_{2}$
                &6  &$4\lambda +\lambda_{1}+\lambda_{2}$\\
7   &$5\lambda+\lambda_{1}+\lambda_{2} $
        &8  &$5\lambda +2\lambda_{1}+\lambda_{2}$
                &9  &$6\lambda +2\lambda_{1}+\lambda_{2}$\\
10  &$7\lambda+2\lambda_{1}+\lambda_{2} $
        &11 &$7\lambda +3\lambda_{1}+\lambda_{2}$
                &12 &$8\lambda +3\lambda_{1}+\lambda_{2}$\\
13  &$9\lambda+3\lambda_{1}+\lambda_{2} $
        &14 &$9\lambda +3\lambda_{1}+\lambda_{2}+2$
                &15 &$10\lambda +3\lambda_{1}+\lambda_{2}+2$\\
    &$=\rho -1$   
        &   &$=\rho +1$   
                &   &$=\rho +\lambda +1$\\
16  &$11\lambda+3\lambda_{1}+\lambda_{2}+2 $
        &17 &$11\lambda +4\lambda_{1}+\lambda_{2}+2$
                &18 &$12\lambda +4\lambda_{1}+\lambda_{2}+2$\\
19  &$13\lambda+4\lambda_{1}+\lambda_{2}+2 $
        &20 &$13\lambda +5\lambda_{1}+\lambda_{2}+2$
                &21 &$14\lambda +5\lambda_{1}+\lambda_{2}+2$\\
22  &$15\lambda+5\lambda_{1}+\lambda_{2}+2 $
        &23 &$15\lambda +5\lambda_{1}+2\lambda_{2}+2$
                &24 &$16\lambda +5\lambda_{1}+2\lambda_{2}+2$\\
25  &$17\lambda+5\lambda_{1}+2\lambda_{2}+2 $
        &26 &$17\lambda +6\lambda_{1}+2\lambda_{2}+2$
                &27 &$18\lambda +6\lambda_{1}+2\lambda_{2}+2$\\
    &$=2\rho-\lambda-\lambda_{1}  $
        &   &$=2\rho-\lambda  $
                &   &$=2\rho $\\
\hline 
\end{tabular} 
\bigskip

\end{table}

As an example, we present the slice spectral sequence for
$S^{8\lambda}\wedge H\mZ$ in Figure~\ref{fig:S8lambda}. In this, as
before, a dot represents a copy of $\mB_{1,0}$, a circle indicates
$\mB_{2,0}$, and underlining increases the second subscript. Here we
also have circled circles, denoting $\mB_{3,0}$, and the box
indicates a copy of $\mZ$. Asterisks represent the dual to the named
Mackey functor.

\begin{figure}[t!h]
\includegraphics[width=.5\textwidth]{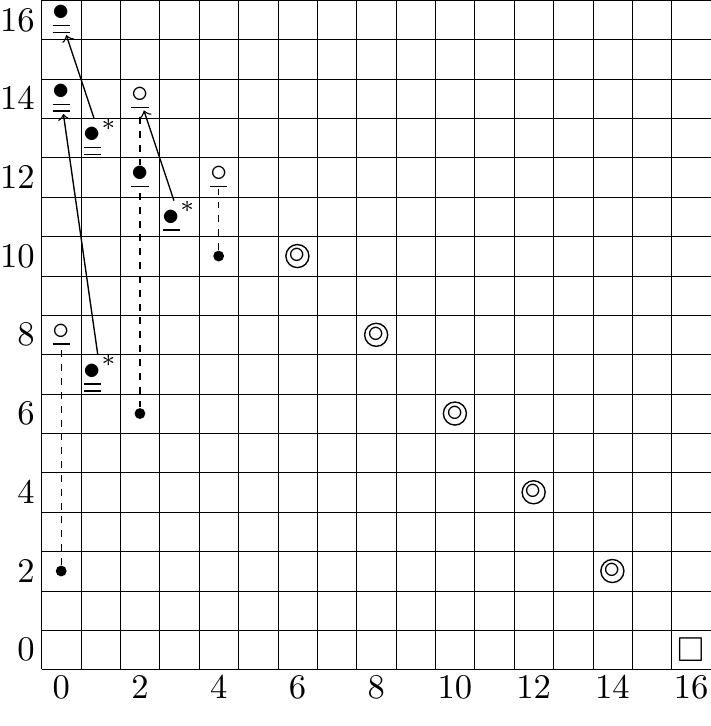}
\caption{The slice spectral sequence for $S^{8\lambda}\wedge H\mZ$ 
and $C_{27}$.}
\label{fig:S8lambda}
\end{figure}

\bibliographystyle{plain}

\bibliography{lambda}

\end{document}